\documentclass[12pt,twoside]{amsart}
\usepackage{amssymb,amsmath,amsthm, amscd, enumerate, mathrsfs}
\usepackage{graphicx, hhline}
\usepackage{mathrsfs}
\usepackage[all]{xy}
\usepackage[dvipdfmx]{hyperref}

\title[Log canonical inversion of adjunction]
{Log canonical inversion of adjunction} 
\author{Osamu Fujino}
\date{2023/8/7, version 0.15}
\keywords{inversion of adjunction, adjunction, log canonical 
singularities, complex analytic spaces}
\address{Department of 
Mathematics, Graduate School of Science, 
Kyoto University, Kyoto 606-8502, Japan}
\email{fujino@math.kyoto-u.ac.jp}
\DeclareMathOperator{\Diff}{Diff}
\newtheorem{thm}{Theorem}[section]
\newtheorem{lem}[thm]{Lemma}
\newtheorem*{claim}{Claim}

\theoremstyle{definition}
\newtheorem{rem}[thm]{Remark}
\newtheorem*{ack}{Acknowledgments}  
\newtheorem{step}{Step}
\newtheorem{say}[thm]{}
\makeatletter
  
  \@addtoreset{equation}{section}
\makeatother
\begin{document}

\maketitle 

\begin{abstract}
This is a short note on the log canonical 
inversion of adjunction. 
\end{abstract}

\section{Introduction}\label{a-sec1}

The following theorem is Kawakita's inversion of adjunction on 
log canonicity (see \cite[Theorem]{kawakita}). 
Although \cite[Theorem]{kawakita} is formulated and proved only for 
algebraic varieties, 
his clever and mysterious 
proof in \cite{kawakita} works in the complex analytic setting. 
Here we will prove it as an application of 
the minimal model theory for projective morphisms 
of complex analytic spaces established in \cite{fujino-minimal} 
following the argument in \cite{hacon} with some suitable 
modifications. 
Our proof is more geometric than Kawakita's. 

\begin{thm}[{Log canonical inversion of adjunction, 
see \cite[Theorem]{kawakita}}]\label{a-thm1.1}
Let $X$ be a normal complex variety and let $S+B$ be an 
effective $\mathbb R$-divisor on $X$ such that 
$K_X+S+B$ is $\mathbb R$-Cartier, $S$ is reduced, and 
$S$ and $B$ have no common irreducible components. 
Let $\nu\colon S^\nu\to S$ be the normalization with 
$K_{S^\nu}+B_{S^\nu}=\nu^*(K_X+S+B)$, where 
$B_{S^\nu}$ denotes Shokurov's different. Then 
$(X, S+B)$ is log canonical in a neighborhood 
of $S$ if and only if $(S^\nu, B_{S^\nu})$ is log 
canonical.  
\end{thm}

We note that $X$ is not necessarily an algebraic 
variety in Theorem \ref{a-thm1.1}. 
It is only a complex analytic space. 
In this note, we will freely use \cite{fujino-minimal} 
and \cite{banica}. 
We assume that the reader is familiar with 
the basic definitions and results of the minimal model 
theory for algebraic varieties (see, 
for example, \cite{kollar-mori}, \cite{bchm}, 
\cite{fujino-fundamental}, \cite{fujino-foundations}, and so on). 

\section{Quick review of the analytic MMP}\label{a-sec2}
In this section, we quickly explain the minimal model theory 
for projective morphisms between complex analytic spaces 
established in \cite{fujino-minimal}. 

\begin{say}[Singularities of pairs]\label{a-say2.1}
As in the algebraic case, we can define {\em{kawamata log terminal pairs}}, 
{\em{log canonical pairs}}, {\em{purely log terminal pairs}}, 
{\em{divisorial log terminal 
pairs}}, and so on, for complex analytic spaces. 
For the details, see \cite[Section 3]{fujino-minimal}. 
\end{say}

One of the main contributions of \cite{fujino-minimal} is 
to find out a suitable complex analytic formulation in order 
to make the original proof of \cite{bchm} 
work with only some minor modifications.

\begin{say}\label{a-say2.2}
Let $\pi\colon X\to Y$ be a projective morphism between 
complex analytic spaces. A compact subset of an analytic 
space is said to be {\em{Stein compact}} 
if it admits a fundamental system of Stein open 
neighborhoods. 
It is well known that if $W$ is a Stein compact 
semianalytic subset of $Y$ then 
$\Gamma (W, \mathcal O_Y)$ is noetherian. 
From now on, we fix a Stein compact subset $W$ of $Y$ such that 
$\Gamma (W, \mathcal O_Y)$ is noetherian. 
Then we can formulate and prove the cone and contraction 
theorem 
over some open neighborhood of $W$ as in the usual 
algebraic case. 
This is essentially due to Nakayama (see \cite{nakayama}). 
We say that $X$ is {\em{$\mathbb Q$-facotrial over $W$}} 
if every prime divisor defined on an open neighborhood of 
$\pi^{-1}(W)$ is $\mathbb Q$-Cartier at any point $x\in \pi^{-1}(W)$. 
Then, in \cite{fujino-minimal}, we show that we can 
translate almost all the results in \cite{bchm} 
into the above analytic setting 
suitably (see \cite[Section 1]{fujino-minimal}). 
\end{say}

Hence we have the minimal model program with ample scaling 
as in the algebraic case. In Section 
\ref{a-sec4}, we will use it in the proof of Theorem 
\ref{a-thm1.1}. 

\begin{say}[Minimal model program with ample scaling]\label{a-say2.3}
Let $(X, \Delta)$ be a divisorial log terminal pair such 
that $X$ is $\mathbb Q$-factorial over $W$ and let $C\geq 0$ 
be a $\pi$-ample 
$\mathbb R$-divisor on $X$ such that 
$(X, \Delta+C)$ is log canonical and that $K_X+\Delta+C$ 
is nef over $W$. 
Then we can run the {\em{$(K_X+\Delta)$-minimal 
model program with scaling of $C$ over $Y$ around $W$}} from 
$(X_0, \Delta_0):=(X, \Delta)$ as 
in the algebraic case. We put $C_0:=C$. 
Thus we get a sequence of flips and 
divisorial contractions 
\begin{equation*}
(X_0, \Delta_0)\overset{\phi_0}{\dashrightarrow} (X_1, 
\Delta_1)
\overset{\phi_1}{\dashrightarrow} \cdots 
\overset{\phi_{i-1}}{\dashrightarrow} 
(X_i, \Delta_i) \overset{\phi_i}{\dashrightarrow} (X_{i+1}, 
\Delta_{i+1})
\overset{\phi_{i+1}}{\dashrightarrow} \cdots 
\end{equation*} 
over $Y$ 
with $\Delta_i:=(\phi_{i-1})_*\Delta_{i-1}$ and  
$C_i:=(\phi_{i-1})_*C_{i-1}$ for every $i\geq 1$. 
We note that each step $\phi_i$ exists only after shrinking 
$Y$ around $W$ suitably. 
We also note that 
\begin{equation*}
\lambda_i:=\inf \{\mu\in \mathbb R_{\geq 0} \mid 
{\text{$K_{X_i}+\Delta_i+\mu C_i$ is nef over $W$}}
\}
\end{equation*} 
and that each step $\phi_i$ is induced by a $(K_{X_i}+\Delta_i)$-negative 
extremal ray $R_i$ such that 
$(K_{X_i}+\Delta_i+\lambda_i C_i)\cdot R_i=0$. 
We have 
\begin{equation*}
\lambda_{-1}:=1\geq \lambda_0\geq \lambda _1\geq \cdots   
\end{equation*} 
such that this sequence is  
\begin{itemize}
\item finite with $\lambda_{N-1}>\lambda_N=0$, or 
\item infinite with $\lim _{i\to \infty} \lambda_i=0$.  
\end{itemize}
Of course, it is conjectured that the above minimal 
model program always terminates after finitely many 
steps. Unfortunately, however, it is still widely open 
even when $\pi\colon X\to Y$ is algebraic. 
\end{say} 

Anyway, for the details of the minimal model theory 
for projective morphisms of complex analytic spaces, 
see \cite{fujino-minimal}. 

\section{Zariski's subspace theorem}\label{a-sec3} 

In this short section, we quickly review 
{\em{Zariski's subspace theorem}} 
following \cite{abhyankar}. 

\begin{say}[{see \cite[(1.1)]{abhyankar}}]\label{a-say3.1}
Let $R_1$ and $R_2$ be noetherian local rings. 
Then we 
say that $R_2$ {\em{dominates}} $R_1$ 
if $R_1$ is a subring of $R_2$ and 
$m_{R_1}\subset m_{R_2}$ holds, where $m_{R_1}$ (resp.~$m_{R_2}$) 
is the maximal ideal of $R_1$ (resp.~$R_2$). 
\end{say}

\begin{say}[{see \cite[(1,1)]{abhyankar}}]\label{a-say3.2} 
Let $R_1$ and $R_2$ be noetherian local rings such that 
$R_1$ is a subring of $R_2$. 
We say that $R_1$ is a {\em{subspace}} of $R_2$ if 
$R_1$ with its Krull topology is a 
subspace of $R_2$ with its Krull topology. 
This means that $R_2$ dominates $R_1$ and 
there exists a sequence of non-negative 
integers $a(n)$ such that $a(n)$ tends to 
infinity with $n$ and $R_1\cap m^n_{R_2}
\subset m^{a(n)}_{R_1}$ holds for 
every $n\geq 0$. 
\end{say}

\begin{say}[{see \cite[(1.1)]{abhyankar}}]\label{a-say3.3} 
Let $R_1$ and $R_2$ be noetherian local domains such that 
$R_1$ is a subring of $R_2$.  
Then $\mathrm{trdeg}_{R_1}R_2$ denotes 
the {\em{transcendence degree}} of the quotient field 
of $R_2$ over 
the quotient field of $R_1$. 
Let $h\colon R_2\to R_2/m_{R_2}$ be the canonical surjection, 
where $m_{R_2}$ is the maximal ideal of $R_2$. 
Let $k$ be the quotient field of $h(R_1)$ in $h(R_2)$. 
Then $\mathrm{trdeg}_kh(R_2)$ is called the {\em{residual 
transcendence degree}} of $R_2$ over $R_1$ and 
is denoted by $\mathrm{restrdeg}_{R_1}R_2$.   
\end{say}

We need the following form of Zariski's subspace theorem. 

\begin{thm}[{see, for example, 
\cite[(10.13).~A form of ZST]{abhyankar}}]\label{a-thm3.4}
Let $R_1$ and $R_2$ be noetherian local domains 
such that $R_1$ is analytically irreducible, $R_2$ dominates $R_1$, 
$\mathrm{trdeg}_{R_1} R_2<\infty$, and $\dim R_1+\mathrm{trdeg}_{R_1}R_2
=\dim R_2+\mathrm{restrdeg} _{R_1} R_2$. 
Then $R_1$ is a subspace of $R_2$. 
\end{thm}

Here we do not prove Theorem \ref{a-thm3.4}. 
For the details, 
see \cite[\S 10]{abhyankar}. 

\section{Proof of Theorem \ref{a-thm1.1}}\label{a-sec4}

Let us prove Theorem \ref{a-thm1.1} following the argument 
in \cite{hacon}, where the log canonical inversion of adjunction 
was established for log canonical 
centers of arbitrary dimension. Our proof given below uses 
Zariski's subspace 
theorem as in \cite{kawakita}. 

\begin{proof}[Proof of Theorem \ref{a-thm1.1}]
In this proof, we will closely follow the argument 
in \cite{hacon} with some 
suitable modifications. 
If $(X, S+B)$ is log canonical in a neighborhood 
of $S$, then it is easy to 
see that $(S^\nu, B_{S^\nu})$ is log canonical by adjunction. 
Therefore, it is sufficient to prove that 
$(X, S+B)$ is log canonical near $S$ under the assumption 
that $(S^\nu, B_{S^\nu})$ is log canonical. 
Without loss of generality, we may assume that $S$ is irreducible.  
We take an arbitrary point $P\in S$. 
We can replace $X$ with a relatively 
compact Stein open neighborhood of $P$ 
since the statement is local. 
From now on, we will freely shrink $X$ around $P$ suitably 
throughout the proof without 
mentioning it explicitly.  
\setcounter{step}{0}
\begin{step}\label{a-step1} In this step, we will see that 
we can reduce the problem to the case where 
$K_X+S+B$ is $\mathbb Q$-Cartier. 

The argument here is more or less well known to the 
experts and is standard 
in the theory of minimal models. 
Hence we will only give a 
sketch of the proof. As usual, we can write 
\begin{equation*}
K_X+S+B=\sum _{p=1}^q r_p(K_X+S+B_p)
\end{equation*} 
such that $K_X+S+B_p$ is $\mathbb Q$-Cartier, 
$0<r_p<1$ for every $p$ with $\sum _{p=1}^qr_p=1$, 
and $(S^\nu, B_p^\nu)$ is log canonical 
for every $p$, 
where $K_{S^\nu}+B_p^\nu=\nu^*(K_X+S+B_p)$. 
Note that if $(X, S+B_p)$ is log canonical 
near $S$ for every $p$ then 
$(X, S+B)$ is log canonical 
in a suitable neighborhood of $S$. 
Therefore, we can replace $(X, S+B)$ with $(X, S+B_p)$ and assume 
that $K_X+S+B$ is $\mathbb Q$-Cartier. 
This is what we wanted. 
\end{step}
\begin{step}\label{a-step2} 
In this step, we will make a good partial resolution of 
singularities of the pair 
$(X, S+B)$ by using the minimal model 
program established in \cite{fujino-minimal} 
(see also Section \ref{a-sec2}). 

Let $W$ be a Stein compact subset of $X$ such that 
$\Gamma(W, \mathcal O_X)$ is noetherian and 
that $W$ contains some open neighborhood of $P$. 
By \cite[Theorem 1.21]{fujino-minimal}, 
we can take a projective bimeromorphic morphism 
$\mu\colon Y\to X$ with 
$K_Y+\Delta_Y=\mu^*(K_X+S+B)$ such 
that 
\begin{itemize}
\item[(i)] $Y$ is $\mathbb Q$-factorial over $W$, 
\item[(ii)] $\Delta_Y$ is effective and 
$\Delta_Y=\sum _j d_j \Delta_j$ is the irreducible decomposition,  
\item[(iii)] the pair 
\begin{equation*}
\left(Y, \Delta'_Y:=\sum _{d_j\leq 1}d_j\Delta_j 
+\sum _{d_j>1}\Delta_j\right)
\end{equation*} is divisorial log terminal, and 
\item[(iv)] every $\mu$-exceptional divisor appears in 
$(\Delta'_Y)^{=1}:=\sum _{d_j\geq 1}\Delta_j$. 
\end{itemize} 
Note that $\mu\colon Y\to X$ is sometimes called a 
dlt blow-up of $(X, S+B)$ in the literature 
(see \cite[Theorem 1.21]{fujino-minimal}). 
We write $\Delta'_Y=T+\Gamma$, where $T$ is the strict 
transform of $S$ and $\Gamma:=\Delta'_Y-T$, and put 
\begin{equation*}
\Sigma:=\Delta_Y-T-\Gamma
=\Delta_Y-\Delta'_Y. 
\end{equation*} 
We take an effective Cartier divisor $E$ on $Y$ such that 
$-E$ is $\mu$-ample and $K_Y+T+\Gamma-E$ is $\mu$-nef over 
$W$. We note that we can choose $E$ such 
that $E$ and $T$ have no common components.
Then we run the $(K_Y+T+\Gamma)$-minimal model 
program with scaling of $-E$ over $X$ around $W$. 
We obtain a sequence of flips and divisorial contractions: 
\begin{equation*}
\begin{split}
(Y, T+\Gamma)&=:(Y_0, T_0+\Gamma_0)\overset{\phi_0}\dashrightarrow
(Y_1, T_1+\Gamma_1)\\ 
&\overset{\phi_1}\dashrightarrow
(Y_2, T_2+\Gamma_2)
\overset{\phi_2}\dashrightarrow 
\cdots \overset{\phi_{i-1}}\dashrightarrow
(Y_i, T_i+\Gamma_i)\overset{\phi_i}\dashrightarrow
\cdots. 
\end{split}
\end{equation*}
Note that each step exists only after shrinking $X$ around 
$W$ suitably. Let $\mu_i\colon Y_i\to X$ be the induced 
morphism. 
For any divisor $G$ on $Y$, we let $G_i$ denote the pushforward 
of $G$ on $Y_i$. We put $\lambda_{-1}:=1$. 
By construction, 
there exists a non-increasing sequence of rational numbers 
$\lambda_i\geq \lambda_{i+1}$ with $i\geq 0$ that is either 
\begin{itemize}
\item finite with $\lambda_{N-1}>\lambda_N=0$, or 
\item infinite with $\lim_{i\to \infty}\lambda_i=0$
\end{itemize} 
such that $K_{Y_i}+T_i+\Gamma_i-\lambda E_i$ 
is nef over $W$ for all $\lambda_{i-1} \geq \lambda 
\geq \lambda_i$. 
Without loss of generality, we may assume that 
each $\phi_i$ is a flip for every $i\geq i_0$ or 
that $i_0=N$, that is, the minimal model program stops at $i_0=N$. 
For any positive rational number $t$, 
there is an effective $\mathbb Q$-divisor $\Theta_t$ 
on $Y$ such that $\Theta_t\sim _{\mathbb Q} \Gamma-tE$ and 
$(Y, T+\Theta_t)$ is purely log terminal with $\lfloor T+\Theta_t\rfloor 
=T$. 
In this case, we see that if $t<\lambda_{i-1}$ then $(Y_i, T_i +\Theta_{t, i})$ 
is purely log terminal. 
In particular, $(Y_i, \Theta_{t,i})$ is kawamata 
log terminal. 
\end{step}
\begin{step}\label{a-step3}
In this step, we will check that $T_i\cap \Sigma_i=\emptyset$ holds for 
every $i$. 

We note that $T_i$ is normal since 
$(Y_i, T_i+\Gamma_i)$ is a divisorial log terminal 
pair. Therefore, 
$\mu_i\colon T_i\to S$ factors through $\nu\colon S^\nu\to 
S$. 
By construction, we have $K_{Y_i}+T_i+\Gamma_i+\Sigma_i=
\mu^*_i(K_X+S+B)$. 
Hence 
\begin{equation}\label{a-eq4.1}
\begin{split}
K_{T_i}+\Diff_{T_i}(\Gamma_i+\Sigma_i)
&:=\left(K_{Y_i}+T_i+\Gamma_i+\Sigma_i\right)|_{T_i}
\\&=(\mu'_i)^*(K_{S^\nu}+B_{S^\nu}) 
\end{split}
\end{equation} 
holds, where $\mu'_i\colon T_i\to S^\nu$. 
Assume that $T_i\cap \Sigma_i$ is not empty. 
Then we see that 
$(T_i, \Diff_{T_i}(\Gamma_i+\Sigma_i))$ 
is not log canonical. 
By \eqref{a-eq4.1}, 
this is a contradiction since $(S^\nu, B_{S^\nu})$ is log 
canonical by assumption. 
This implies that $T_i\cap \Sigma_i=\emptyset$ holds 
for every $i$. In particular, 
we have 
\begin{equation*}
\begin{split}
K_{T_i}+\Diff_{T_i}(\Gamma_i+\Sigma_i)
&=\left(K_{Y_i}+T_i+\Gamma_i+\Sigma_i\right)|_{T_i}
\\ &=\left(K_{Y_i}+T_i+\Gamma_i\right)|_{T_i}
=:K_{T_i}+\Diff_{T_i}(\Gamma_i). 
\end{split}
\end{equation*}
\end{step}

\begin{step}\label{a-step4}
In this step, we will show that $\phi_i|_{T_i}\colon 
T_i\dashrightarrow T_{i+1}$ is an isomorphism 
for every $i$. 
Moreover, we will prove that if $\phi_i$ is a flip then 
$\phi_i$ is an isomorphism on some open neighborhood of $T_i$. 

First, we assume that $\phi_i$ is a flip. 
We consider the following 
flipping diagram 
\begin{equation*}
\xymatrix{
(Y_i, T_i+\Gamma_i) \ar@{-->}[rr]^-{\phi_i}
\ar[dr]_-{\varphi_i}& 
&(Y_{i+1}, T_{i+1}+\Gamma_{i+1})\ar[dl]^-{\varphi_i^+}\\\
& Z_i  &
}
\end{equation*} 
and we let $W_i$ denote the normalization of $\varphi_i(T_i)$. 
Let $C$ be any flipping curve. 
If $C$ is contained in $T_i$, 
then we obtain 
\begin{equation}\label{a-eq4.2}
(K_{Y_i}+T_i+\Gamma_i)\cdot 
C=(K_{Y_i}+T_i+\Gamma_i+\Sigma_i)\cdot 
C=0
\end{equation}
since $T_i\cap \Sigma_i=\emptyset$ by Step \ref{a-step3}. 
This is absurd. 
Hence this implies that the natural map $T_i \to W_i$ is an isomorphism. 
By the same argument, we see that the natural 
map $T_{i+1}\to W_i$ is 
also an isomorphism. 
This means that $\phi_i|_{T_i}\colon 
T_i\dashrightarrow T_{i+1}$ is an isomorphism when $\phi_i$ is a flip. 
By the above argument, we see that $T_{i+1}$ (resp.~$T_i$) does not contain 
any flipped (resp.~flipping) curves. Note that if $T_i\cdot C>0$ holds for 
some flipping curve $C$ then $-T_{i+1}$ is $\varphi^+_i$-ample. 
Hence $T_i$ is disjoint from the flipping locus. This implies that 
$\phi_i$ is an isomorphism near $T_i$ when $\phi_i$ is a flip. 

Next, we assume that $\phi_i$ is a divisorial contraction. 
In this case, $\phi_i|_{T_i}\colon T_i\dashrightarrow T_{i+1}$ is obviously 
a projective bimeromorphic morphism between normal complex varieties. 
Let $C$ be any curve contracted by $\phi_i$. 
Assume that $C$ is contained in $T_i$. 
Then, by the same computation as in 
\eqref{a-eq4.2}, we get a contradiction. 
This means that $\phi_i|_{T_i}\colon T_i\to T_{i+1}$ 
does not contract any curves. 
Thus, $\phi_i|_{T_i}\colon T_i\dashrightarrow T_{i+1}$  
is an isomorphism. 

We get the desired statement. 
\end{step}

\begin{step}\label{a-step5}
In this step, we will prove that the natural restriction 
map  
\begin{equation*}
(\mu_{i_0})_*\mathcal O_{Y_{i_0}}(-m\Sigma_{i_0}-aE_{i_0})
\to (\mu_{i_0})_*\mathcal O_{T_{i_0}}(-aE_{i_0})
\end{equation*}
is surjective over some open neighborhood of 
$P$ for every positive integer 
$m\geq a/\lambda_{i_0-1}$ such that $m\Sigma$ is 
an integral divisor, where $a$ is the smallest positive integer 
such that $aE_{i_0}$ is Cartier.  

By definition, $aE_{i_0}$ is Cartier. By Step \ref{a-step4}, 
$Y_{i_0}\dashrightarrow Y_i$ is an isomorphism 
on some open neighborhood of $T_{i_0}$ for every $i\geq i_0$. 
Therefore, 
$aE_i$ is Cartier on some open neighborhood 
of $T_i$ for every $i\geq i_0$. 
Since $(Y_i, T_i+\Gamma_i)$ is divisorial log terminal and $T_i$ 
is a $\mathbb Q$-Cartier integral divisor, we have 
the following short exact sequence: 
\begin{equation}\label{a-eq4.3}
\begin{split}
0&\to \mathcal O_{Y_i}(-m\Sigma_i-aE_i-T_i)\to 
\mathcal O_{Y_i}(-m\Sigma_i -aE_i)\\ &\to 
\mathcal O_{T_i}(-aE_i)\to 0 
\end{split}
\end{equation} 
for every $i\geq i_0$ and every $m$ such that $m\Sigma_i$ 
is integral (cf.~\cite[Proposition 5.26]{kollar-mori}). 
Here, we used the fact that $T_i\cap \Sigma_i=\emptyset$ 
(see Step \ref{a-step3}). 
Let $U$ be an open neighborhood of $P$ contained in $W$. 
For every positive integer $m\geq a$ such that $m\Sigma$ is an integral 
divisor, there exists $i$ such that $\lambda_{i-1}\geq a/m \geq \lambda_i$. 
If further $m\geq a/\lambda_{i_0-1}$, 
then $i\geq i_0$. Since 
\begin{equation*}
-m\Sigma_i-aE_i-T_i\sim _{\mathbb Q, \mu_i} K_{Y_i}+\Theta_{\frac{a}{m}, i} 
+(m-1) \left(K_{Y_i}+T_i+\Gamma_i -\frac{a}{m}E_i\right), 
\end{equation*} 
$\left(Y_i, \Theta_{\frac{a}{m}, i}\right)$ is kawamata log terminal, 
$K_{Y_i}+T_i+\Gamma_i-\frac{a}{m}E_i$ is nef over $U$, 
we obtain that 
\begin{equation}\label{a-eq4.4}
R^1(\mu_i)_*\mathcal O_{Y_i}(-m\Sigma_i -aE_i-T_i)=0
\end{equation} 
on $U$ by the Kawamata--Viehweg vanishing theorem 
for projective bimeromorphic morphisms of 
complex analytic spaces. 
Hence the natural restriction map 
\begin{equation*}
(\mu_i)_*\mathcal O_{Y_i}(-m\Sigma_i-aE_i)\to 
(\mu_i)_*\mathcal O_{T_i}(-m\Sigma_i-aE_i)
=(\mu_i)_*\mathcal O_{T_i}(-aE_i)
\end{equation*} 
is surjective on $U$ by 
\eqref{a-eq4.3} and \eqref{a-eq4.4}. 
Note that 
\begin{equation*} 
(\mu_i)_*\mathcal O_{Y_i}(-m\Sigma_i-aE_i)
=(\mu_{i_0})_*\mathcal O_{Y_{i_0}}(-m\Sigma_{i_0}-aE_{i_0})
\end{equation*}
and 
\begin{equation*}
(\mu_i)_*\mathcal O_{T_i}(-aE_i)=
(\mu_{i_0})_*\mathcal O_{T_{i_0}}(-aE_{i_0})
\end{equation*}
hold because $Y_{i_0}\dashrightarrow Y_i$ 
is an isomorphism in codimension one and $Y_{i_0}\dashrightarrow 
Y_i$ is an isomorphism on some open neighborhood 
of $T_{i_0}$ by Step \ref{a-step4}, respectively. 
Thus, the natural restriction map 
\begin{equation}\label{a-eq4.5}
(\mu_{i_0})_*\mathcal O_{Y_{i_0}}(-m\Sigma_{i_0}-aE_{i_0})
\to (\mu_{i_0})_*\mathcal O_{T_{i_0}}(-aE_{i_0})
\end{equation} 
is surjective on $U$ for every positive integer $m\geq 
a/\lambda_{i_0-1}$ such that 
$m\Sigma$ is an integral divisor. 
This is what we wanted. 
\end{step}

\begin{step}\label{a-step6} 
In this final step, we will get a contradiction 
by assuming that $(X, S+B)$ is not log canonical at $P$. 
Here, we will use Zariski's subspace theorem as in \cite{kawakita}. 

The assumption implies that $P\in \mu(\Sigma)$. 
Note that the non-log canonical locus of $(X, S+B)$ is 
$\mu(\Sigma)$ set theoretically. 
By construction, $(Y_i, T_i+\Gamma_i)$ is divisorial log 
terminal. Therefore, the non-log canonical locus of $(Y_i, T_i+\Gamma_i+\Sigma_i)$ 
is nothing but the support of $\Sigma_i$. 
Therefore, $\mu(\Sigma)=\mu_i(\Sigma_i)$ holds set theoretically for 
every $i$. 
Hence we have $P\in \mu_{i_0}(\Sigma_{i_0})$. 
\begin{claim}\label{a-claim}
Let $\mathcal O_{X,P}$ be the localization of $\mathcal O_X$ at $P$ and 
let $m_P$ denote the maximal ideal of $\mathcal O_{X, P}$. 
For every positive integer $n$, 
there exists a divisible positive integer $\nu(n)$ such that 
\begin{equation*}
(\mu_{i_0})_*\mathcal O_{Y_{i_0}}(-\nu(n)\Sigma_{i_0}-aE_{i_0})_P 
\subset m^n_P\subset \mathcal O_{X, P}
\end{equation*} 
holds, where 
$(\mu_{i_0})_*\mathcal O_{Y_{i_0}}(-\nu(n)\Sigma_{i_0}-aE_{i_0})_P$ 
denotes the localization of 
$(\mu_{i_0})_*\mathcal O_{Y_{i_0}}(-\nu(n)\Sigma_{i_0}-aE_{i_0})$ 
at $P$. 
\end{claim}
\begin{proof}[Proof of Claim] 
We take $Q\in \Sigma_{i_0}$ such that 
$\mu_{i_0}(Q)=P$. We consider 
$\mathcal O_{X, P}\hookrightarrow 
\mathcal O_{Y_{i_0}, Q}$, where 
$\mathcal O_{Y_{i_0}, Q}$ is 
the localization of $\mathcal O_{Y_{i_0}}$ at $Q$. 
It is well known that $\mathcal O_{X, P}$ is excellent. 
Therefore, $\mathcal O_{X, P}$ is analytically irreducible 
since $X$ is normal. 
Since $\mu_{i_0}\colon Y_{i_0}\to X$ is a projective 
bimeromorphic morphism, the quotient field of 
$\mathcal O_{Y_{i_0}, Q}$ coincides 
with the one of $\mathcal O_{X, P}$. We note 
that the natural map $\mathcal O_{X, P}\to \mathcal O_{Y_{i_0}, Q}/m_Q$ 
is surjective, where $m_Q$ is the maximal ideal of $\mathcal O_{Y_{i_0}, Q}$. 
Hence we can use Zariski's subspace theorem (see 
Theorem \ref{a-thm3.4}). 
Thus we get a large and divisible positive integer $\nu(n)$ 
with the desired property. 
\end{proof}
We consider the localization of 
the following restriction map 
$\mathcal O_X\simeq 
(\mu_{i_0})_*\mathcal O_{Y_{i_0}}\to (\mu_{i_0})_*\mathcal O_{T_{i_0}}$ 
at $P$. 
We put $A=\mathcal O_{X, P}$, $M=\left((\mu_{i_0})_*\mathcal O_{T_{i_0}}
\right)_P$, and $N=\left((\mu_{i_0})_*\mathcal O_{T_{i_0}}(-aE_{i_0})
\right)_P$. 
Then, by the surjection \eqref{a-eq4.5} in Step \ref{a-step5} and 
Claim, we obtain 
that $N=(0)$ by Lemma \ref{a-lem4.1} below. This is a contradiction. 
\end{step} 
Hence, we obtain that $(X, S+B)$ is log canonical at $P$. 
Since $P$ is an arbitrary point of $S$, $(X, S+B)$ is 
log canonical in a neighborhood of $S$. 
We finish the proof of Theorem \ref{a-thm1.1}. 
\end{proof} 

We used the following easy commutative algebra lemma 
in the above proof of Theorem \ref{a-thm1.1}. 

\begin{lem}\label{a-lem4.1}
Let $(A, \mathfrak m)$ be a noetherian 
local ring, let $M$ be a finitely generated $A$-module, and let 
$\varphi\colon A\to M$ be a homomorphism of 
$A$-modules. Let $I_1\supset I_2\supset \cdots 
\supset I_k\supset \cdots$ be a chain of 
ideals of $A$ such that there exists $\nu(n)$ satisfying 
$I_{\nu(n)}\subset \mathfrak m^n$ for every positive integer $n$. 
Let $N$ be an $A$-submodule of $M$. 
Assume that $\varphi(I_k)=N$ holds for every positive integer $k$. 
Then we have $N=(0)$. 
\end{lem}

\begin{proof}
Let $b$ be any element of $N$. 
Then we can take $a\in I_{\nu(n)}\subset 
\mathfrak m^n$ such that $\varphi(a)=b$. 
This implies that $b=\varphi(a)\in \mathfrak m^nM$. 
Hence $b\in \mathfrak m^nM$ holds for every 
positive integer $n$. 
Thus we obtain $b\in \bigcap _n \mathfrak m^nM=(0)$. 
Therefore, $b=0$ holds, that is, 
$N=(0)$. 
\end{proof}

We close this short note with a remark. 

\begin{rem}\label{a-rem4.2} 
If $(X, S+B)$ is algebraic in Theorem \ref{a-thm1.1}, 
then we do not need \cite{fujino-minimal}. 
It is sufficient to use the minimal model program at 
the level of \cite{bchm}, the well-known relative Kawamata--Viehweg 
vanishing theorem, and Zariski's subspace theorem 
(see, for example, \cite[(10.6)]{abhyankar}). 
Our proof given here is longer than Kawakita's one 
(see \cite{kawakita}). 
However, it looks more accessible 
for the experts of the minimal model program 
since the argument is more or less standard.
\end{rem}

\begin{ack}\label{a-ack} 
The author would like to thank Masayuki Kawakita very much for answering 
his questions. He also would like to thank 
Shunsuke Takagi very much for answering his questions and giving 
him many fruitful comments. 
Finally, he thanks the referee for useful suggestions and 
comments. 
He was partially 
supported by JSPS KAKENHI Grant Numbers 
JP19H01787, JP20H00111, JP21H00974, JP21H04994. 
\end{ack}


\end{document}